\documentclass[11pt,a4paper]{article}

\usepackage[utf8]{inputenc}
\usepackage{amsmath,amsfonts,amssymb,amsthm,accents,dsfont,bbm}
\usepackage{graphicx,calc}
\usepackage{setspace,tikz,mathrsfs}
\usepackage[top=3cm,right=3cm,left=3cm,bottom=3cm]{geometry}
\usepackage[shortlabels]{enumitem}
\usepackage{chngcntr}
\usepackage[nottoc]{tocbibind}
\usepackage{url}
\usepackage{mathtools}

\author{Anthony H. Dooley and Kieran Jarrett}

\newtheoremstyle{standard}
  {11pt} 
  {} 
  {\itshape} 
  {} 
  {\bfseries} 
  {} 
  {.5em} 
  {} 
  
\theoremstyle{standard}
\newtheorem{theorem}{Theorem}[section]
\newtheorem{lem}[theorem]{Lemma}
\newtheorem{cor}[theorem]{Corollary}
\newtheorem{defn}[theorem]{Definition}
\newtheorem{prop}[theorem]{Proposition}

\newtheorem*{theorem*}{Theorem}

\newtheoremstyle{rem}
  {9pt} 
  {} 
  {} 
  {} 
  {\bfseries} 
  {} 
  {.5em} 
  {} 
\theoremstyle{rem}

\newtheoremstyle{rems}
  {9pt} 
  {} 
  {} 
  {} 
  {\bfseries} 
  {} 
  {.5em} 
  {} 

\theoremstyle{rems}

\newtheoremstyle{e.g.}
  {9pt} 
  {} 
  {} 
  {} 
  {\itshape} 
  {} 
  {.5em} 
  {} 
\theoremstyle{e.g.}
\newtheorem*{e.g.}{E.g.}

\onehalfspace

\SetLabelAlign{Center}{\hfil#1\hfil}

\newcommand{\ind}{\textbf{1}}

\newcommand{\Zd}{\mathbb{Z}^{d}}
\newcommand{\Zto}{\mathbb{Z}^{2}}
\DeclarePairedDelimiter\ceil{\lceil}{\rceil}
\newcommand{\rad}{\textup{rad\,}}
\newcommand{\rmin}{\textup{rmin\,}}
\newcommand{\rmax}{\textup{rmax\,}}
\newcommand{\cdim}{\textup{cdim}}
\DeclarePairedDelimiter\floor{\lfloor}{\rfloor}

\begin{document}

\title{Non-singular $\mathbb{Z}^d$-actions: an ergodic theorem over rectangles with application to the critical dimensions}
\date{}
\maketitle

\begin{abstract}
We adapt techniques of Hochman to prove a non-singular ergodic theorem for $\mathbb{Z}^d$-actions where the sums are over rectangles with side lengths increasing at arbitrary rates, and in particular are not necessarily balls of a norm. This result is applied to show that the critical dimensions with respect to sequences of such rectangles are invariants of metric isomorphism. These invariants are calculated for a class of product actions.
\end{abstract}

\section{Introduction}

Let $G$ be a countable group with a non-singular action on a standard probability space $(X,\mathcal{B},\mu)$, which is assumed to be ergodic throughout. Each $g \in G$ induces a non-singular map on $X$ which we also denote by $g$. The measures $\mu$ and $\mu \circ g$ are equivalent and so the Radon-Nikodým derivative
$$\omega_g = \frac{d \mu \circ g}{d \mu}$$
is well defined and strictly positive almost everywhere. In turn each $g \in G$ induces an isometric linear action on $L^1$ given by $\hat{g}\phi(x) = \phi(gx)\omega_g(x)$.

For conservative integer actions the Hurewicz ergodic theorem states that for $\phi \in L^1$ 
$$\lim_{n \rightarrow \infty} \frac{\sum_{i=1}^n \phi(ix)\omega_i(x)}{\sum_{i=1}^n \omega_i(x)} = \int \phi \, d\mu$$
almost everywhere. Since the action is conservative if $\phi > 0$ a.s. then both the numerator and the denominator in the theorem diverge to infinity. Therefore the ergodic theorem says that both are diverging at the same rate. This suggests that the growth rate of $\sum_{i=1}^n \omega_i$ may encode some intrinsic behaviour of the system. This motivated work by Dooley and Mortiss \cite{DoolMort2,DoolMort3,DoolMort1,Mort1} where they conducted a rigorous study of the growth rate of $\sum_{i=1}^n \omega_i$ and created invariants called the upper and lower critical dimensions.We aim to extend this study from the context of $\mathbb{Z}$-actions to those of other countable groups, with $\mathbb{Z}^d$-actions being the focus of this paper. The critical dimensions are defined for a countable group $G$ as follows. 

Fix a sequence $e \in B_1 \subseteq B_2 \subseteq ...$ of finite subsets of $G$, we will refer to such a sequence as a \textit{summing sequence}. For $t \in \mathbb{R}$ write
\begin{align*} 
L_t &= \left\lbrace x \in X : \liminf_{n \to \infty} \frac{1}{| B_n |^t} \sum_{g \in B_n} \omega_g(x) > 0 \right\rbrace
\end{align*}
and
\begin{align*}
U_t &= \left\lbrace x \in X : \limsup_{n \to \infty} \frac{1}{| B_n |^t} \sum_{g \in B_n} \omega_g(x) = 0 \right\rbrace.
\end{align*}
Observe that $L_t$ and $U_t$ are decreasing and increasing respectively with $t$, and are disjoint.

\begin{defn}
The \emph{lower critical dimension} of $(X,\mu,G)$ with respect to summing sequence $\lbrace B_n \rbrace_{n=1}^\infty$ is defined by $$\alpha = \alpha(B_n) = \sup \lbrace t : \mu(L_t) = 1 \rbrace.$$
The \emph{upper critical dimension} of $(X,\mu,G)$ with respect to $\lbrace B_n \rbrace_{n=1}^\infty$ is defined by $$\beta = \beta(B_n) = \inf \lbrace t : \mu(U_t) = 1 \rbrace.$$
When $\alpha$ and $\beta$ coincide we let $\gamma = \alpha = \beta$ and call it the \emph{critical dimension}.
\end{defn}

Intuitively, the lower critical dimension gives the slowest growth rate of all the subsequences of $\sum_{g \in B_n}\omega_g(x)$, and the upper critical dimension the fastest. It follows from the definitions that $0 \leq \alpha \leq \beta$ and from Fatou's lemma that $\alpha \leq 1$.

When $G = \mathbb{Z}$ the sets $B_n$ are normally taken to be the intervals $[1,n]$, considered as a subset of $\mathbb{Z}$, in analogy with range of the sums in the ergodic theorem. However, in the case of a general countable group there is not such a clear choice for $B_n$. This raises the question: how does the choice of the summing sequence affect the critical dimensions? 

We start to address this question in section \ref{prodact} by examining product $\mathbb{Z}^d$-actions on spaces $X = X_1 \times ... \times X_d$, where each $X_i$ has an associated transformation $T_i$. We consider the critical dimensions with respect to sequences of \textit{rectangles} ${B_n = B_n^1 \times ... \times B_n^d}$ where each $B_n^i = [-s_i(n),s_i(n)]$ for some increasing ${s_i : \mathbb{N}_0 \rightarrow \mathbb{N}_0}$. Note that we are requiring rectangles to be symmetric about the origin. For each $1 \leq i \leq d$ we write $\gamma_i$ for the single critical dimension (if it exists) of $T_i$ with respect to $[-n,n]$. Our main result in this section, theorem \ref{CDprodact}, shows that for these actions the critical dimension $\gamma(B_n)$ of the product action is a weighted average of the $\gamma_i$, with weightings determined by relative growth rates of the functions $s_i$. 

\newpage 

\begin{theorem} \label{=case}
Suppose that for function $s : \mathbb{N} \rightarrow \mathbb{N}$ the limits $c_i = \lim_{n \rightarrow \infty} \frac{\log{s_i(n)}}{\log{s(n)}}$ exist, and that one of these is non-zero. Then 
$$\gamma(B_n) = \frac{\sum_{i=1}^d c_i \gamma_i}{\sum_{i=1}^d c_i}.$$
\end{theorem}

A pair of illustrative applications of this result in the case $d = 2$ are that
\begin{align*}
\gamma([-n,n] \times [-n^2,n^2]) = \frac{\gamma_1 + 2 \gamma_2}{3} & & \text{and} & &
\gamma([-n,n] \times [-\floor{e^n-1},\floor{e^n-1}]) = \gamma_2.
\end{align*}

For integer actions, the first and simplest demonstration of the intrinsic nature of the critical dimensions is due to Mortiss who proved that when $B_n = [1,n]$ they are invariants of metric isomorphism \cite{Mort1}. One of the purposes of this paper is to show that the same argument holds for $\mathbb{Z}^d$-actions with the $B_n$ taken to be rectangles. However, Mortiss' argument made use of the ergodic theorem. 

Given an action of a group $G$ on a finite measure space $(X,\mu)$ and a summing sequence ${B_1 \subseteq B_2 \subseteq ...}$ of finite subsets of $G$ the ergodic theorem is satisfied if for every integrable function $\phi$
\begin{align*}
\lim_{n \rightarrow \infty} \frac{\sum_{g \in B_n} \hat{g}\phi}{\sum_{g \in B_n} \hat{g}1} = \int \phi \, d\mu
\end{align*}
almost everywhere. 

For non-singular actions of countable groups the question of when the ergodic theorem holds is an area of current research. The foremost positive result is due to Hochman \cite{Hoch1}, who proved it holds for free, non-singular and ergodic $\mathbb{Z}^d$-actions and $B_n = \lbrace u \in \Zd : \|u\| \leq n \rbrace$ where $\| \cdot \|$ is a norm on $\mathbb{R}^d$. Crucially, this does not include the case where the $B_n$ are rectangles, as described above, because the $s_i(n)$ may have completely different orders of growth. To apply the arguments of Mortiss verbatim it is therefore necessary to show the ergodic theorem extends to rectangles. This requires care because there are natural choices of $B_n$ for which the ergodic theorem is known to fail. One such, due to Brunel and Krengel \cite{Kren1}, shows the ratio ergodic theorem (a consequence of the ergodic theorem in this context) fails for $B_n = [0,n]^d$ and $d > 1$. The generally cited reason for this failure is that the sets $[0,n]^d$ fail to satisfy the Besicovitch property, as defined in \cite{Hoch1}. However, as noted in \cite{Guz1}, sequences of rectangles with increasing side lengths have the Besicovitch property.

Prior to Hochman's work, Feldman \cite{Fel1} used a simpler method to prove a weaker result; it assumed that each of the standard generators $e_1,...,e_n$ of $\mathbb{Z}^d$ acted conservatively on $X$ and took $\| \cdot \|$ to be the sup-norm, and but otherwise unchanged. Both methods follow the standard approach: one produces a dense subset of $L^1$ for which the theorem holds and then applies a maximal inequality to extend this to all of $L^1$. 

Upon a quick examination of Feldman's proof it becomes apparent that the sets $B_n = [-n,n]^d$ can be replaced by the rectangles $\prod_{i=1}^d [-s_i(n),s_i(n)]$ to produce an appropriate dense set of functions. The maximal inequality is then proved using two key properties. The first is that balls of norms in $\mathbb{R}^d$ satisfy the Besicovitch property, see \cite{Guz1} for a proof. The second is that they satisfy the doubling condition, i.e. $|B_{2n}| \leq C |B_n|$ for some fixed constant $C$. We have already noted our rectangles satisfy the Besicovitch property. Moreover, rectangles $B_n$ satisfy an additive version of doubling condition, i.e.
\begin{align} \label{mDC}
|2B_n| = |B_n + B_n| \leq 2^d |B_n|
\end{align}
where for rectangles $B_n$ and $\lambda \in \mathbb{N}$ we let
$\lambda B_n = \prod_{i=1}^d [-\lambda s_i(n),\lambda s_i(n)]$.
This coincides with the sum of $\lambda$ copies of $B_n$.

By modifying the proof of the maximal inequality in \cite{Fel1} to use (\ref{mDC}) rather than the doubling condition for metrics it can be seen that the maximal inequality holds for rectangles. This means that Feldman's result can be extended so that the sums can be taken over rectangles. We explain this modification in section 2.

It is then natural to ask whether similar changes can be made to Hochman's method of producing a dense set of functions. His approach consistently views $\Zd$ as translation invariant metric space, and so we make a light assumption that our rectangles are balls of \textit{rectangular metrics}, see (\ref{recmet}) for the details. It also makes use of both of the doubling and Besicovitch properties to produce the appropriate dense set of functions, in addition to a type of finite dimensionality property of $\mathbb{Z}^d$ with respect to balls of norms. In section \ref{ergthm} we will set out how one can again replace the standard doubling condition with (\ref{mDC}). We also show that $\Zd$ satisfies a corresponding finite dimensionality property with respect to these rectangles. Taken together, these allow us to adapt Hochman's method to prove an ergodic theorem.

\begin{theorem} \label{recET}
Let $\mathbb{Z}^d$ have a non-singular and ergodic action on a probability space $(X,\mu)$ and $B_n = \lbrace u \in \mathbb{Z}^d : \rho(u,0) \leq n \rbrace$ for some rectangular metric $\rho$ on $\mathbb{Z}^d$. Then for every $\phi \in L^1$ as $n \rightarrow \infty$
\begin{align*}
\frac{\sum_{u \in B_n} \hat{u}\phi}{\sum_{u \in B_n} \hat{u}1} \rightarrow \int \phi \, d\mu.
\end{align*}
\end{theorem}

With this result in hand, the arguments of Mortiss can be applied to see the critical dimensions of summing sequences of rectangles are invariants of metric isomorphism.

\begin{cor} \label{CDinvar}
The upper and lower critical dimensions with respect to any summing sequence of balls $B_n = \lbrace u \in \mathbb{Z}^d : \rho(u,0) \leq n \rbrace$, for some rectangular metric $\rho$, are invariants of metric isomorphism.
\end{cor}

\section{The ergodic theorem for rectangles} \label{ergthm}

In the standard proof for ergodic theorems there are two key ingredients. The first is a maximal inequality. For $\phi \in L^1(X)$ let 
$$R_n\phi(x) = \frac{\sum_{g \in B_n} \phi(gx)\omega_g(x) }{\sum_{g \in B_n} \omega_g(x)}.$$
The maximal inequality holds if there exits $C > 0$ such that for any $\phi \in L^1$ and $\epsilon > 0$
$$\mu\left( \sup_{n \geq 1} \left| R_n\phi \right| > \epsilon \right) \leq \frac{C}{\epsilon} \|\phi\|_1.$$
The second key ingredient there is dense subset $H$ of $L^1$ such that for all $h \in H$ and all $\sigma \in G$
\begin{align} \label{cnv}
\frac{\sum_{g \in B_n \setminus \sigma B_n} \hat{g}h - \sum_{g \in \sigma B_n \setminus  B_n} \hat{g}h}{\sum_{g \in B_n} \hat{g}h} \to 0 
\end{align}
almost surely.

With these in hand the ergodic theorem can be proved as follows. Consider the set
$$D = \lbrace c + h - \hat{\sigma}h : c \in \mathbb{R}, \sigma \in G, h \in H \rbrace.$$
Using a standard argument laid out in \cite{Aar1}, which uses the density of $H$, one can see that $D$ is dense in $L^1$. Moreover
$$\frac{\sum_{g \in B_n} \hat{g}(c+h - \hat{\sigma}h)}{\sum_{g \in B_n} \hat{g}1} 
= c + \frac{\sum_{g \in B_n \setminus \sigma B_n} \hat{g}h - \sum_{g \in \sigma B_n\setminus  B_n} \hat{g}h}{\sum_{g \in B_n} \hat{g}1}$$ 
Therefore by (\ref{cnv}) the left hand side converges to $c$ almost surely. 

Choose $c_m + h_m - \hat{v}h_m \in D$ such that $\| \phi - c_m + h_m - \hat{v}h_m \|_1 < m^{-1}$ then by the above combined with the maximal inequality
$$\mu \left( \limsup_{n \to \infty} |R_n\phi - c_m| > \epsilon \right) =  \mu \left( \limsup_{n \to \infty} |R_n(\phi - c_m + h_m - \hat{v}h_m)| > \epsilon \right) \leq \frac{C}{\epsilon m}.$$
Note that $c_m = \int c_m + h_m - \hat{v}h_m \,d\mu \to \int \phi \, d\mu$ as $m \to \infty$. Hence choosing $m$ large enough for $|c_m - \int \phi \, d\mu| < \epsilon$ we see that
$$\mu \left( \limsup_{n \to \infty} |R_n\phi - c| > 2\epsilon \right) \leq \frac{C}{\epsilon m}$$
for all $m$ sufficiently large. Therefore the left hand side is $0$ for all $\epsilon > 0$, which proves the theorem.

In the case where $H = L^\infty$ condition (\ref{cnv}) is implied by 
\begin{align*}
\frac{\sum_{g \in B_n \triangle \sigma B_n} \omega_g}{\sum_{g \in B_n} \omega_g} \rightarrow 0 \qquad \text{a.s.} \tag{nsFC} \label{nsFC}
\end{align*}
which we call the \textit{non-singular F{\o}lner Condition}. In the case that the action is measure preserving this reduces to the standard F{\o}lner condition for the sequence $B_n$, implying that $G$ is amenable. For integer actions if $B_n = [1,n]$ then (nsFC) follows from the Chacon-Ornstein lemma, see for example \cite{Aar1}, and the assumption that the action is conservative. Hochman's variant of the Chacon-Ornstein lemma in \cite{Hoch1}, summing over balls of norms, also implies (nsFC). 

It should be noted that Feldman's argument shows (\ref{cnv}) directly for a smaller set that $L^\infty$, rather than via (nsFC).

We wish to consider the sums over balls $B_n$ of metrics on $\Zd$ which take the form 
\begin{align} \label{recmet}
\rho(u,v) = \max_{1 \leq i \leq d} F_i(|u_i - v_i|)
\end{align}
where each $F_i : [0,\infty) \to [0,\infty)$ (as subsets of $\mathbb{R}$) satisfies $F_i(0) = 0$ and is subadditive; these properties ensure $d$ is a metric. We also assume that each $F_i$ is strictly increasing, so has an inverse which we denote by $f_i$. It follows that $f_i$ is superadditive on $[0,\infty)$, and hence $f_i(n) \geq n f_i(1)$ for all $n \in \mathbb{N}$. 

If $\rho$ satisfies these conditions then we will refer to it as a \textit{rectangular metric on $\mathbb{Z}^d$}. For a subset $S \subseteq \Zd$ with metric $\rho$ we say $\rho$ is a \textit{rectangular metric on $S$} if it is the restriction of a rectangular metric on $\Zd$ to $S$. In this case we say $(S,\rho)$ is a \textit{rectangular metric space}. 

We will refer to the balls $B_r(z)$ of rectangular metrics as \textit{rectangles} which we assume carry the information of their centre and radius with them. A crucial property to note is that rectangular metrics are translation invariant. This means $B_r(z) = z + B_r$ where 
$$B_r = B_r(0) = \prod_{i=1}^d [-\floor{f_i(r)},\floor{f_i(r)}].$$ We will mainly be focussed on rectangles with $r = n \in \mathbb{N}_0$. We call a summing sequence $B_1 \subseteq B_2 \subseteq ...$ \textit{rectangular} if it is constructed in this way for some rectangular metric.

Restricting temporarily to $d = 2$ some examples of rectangular metrics are given by $F_1(s) = s$ and $F_2(s) = \log{(1+s)}$ or $F_2(s) = \sqrt{s}$. These metrics have the balls ${[-n,n] \times [-\floor{e^n-1},\floor{e^n-1}]}$ and $[-n,n] \times [-n^2,n^2]$ respectively. 

Observe that neither of these are sequences of balls in a fixed norm and hence are not covered by Hochman's result. However, the techniques used to tackle balls of norms can be adapted to rectangles. We start by arguing that the maximal inequality holds for the rectangles $B_n = \lbrace u \in \mathbb{Z}^d : \rho(u,0) \leq n \rbrace$ with $C = 4^d$ and then adapt Hochman's proof of (\ref{nsFC}) for these sets.

\subsection{The Besicovitch property and maximal inequality for rectangles}

We begin by recalling some terminology from \cite{Hoch1} and showing rectangular metric spaces have the Besicovitch property (which also follows from a comment in \cite[p.\ 7]{Guz1}), before proving the first disjointification lemma.

Let $(S,\rho)$ be a general metric space. A finite family of balls $\mathcal{U} = \lbrace B_{r(i)}(z_i) \rbrace_{i=1}^N $ is called a \textit{carpet over} $\lbrace z_1,...,z_N \rbrace \subseteq S$. We say a collection of sets has multiplicity $\leq M$ if every point is contained in at most $M$ elements of the collection. A metric space has the \textit{Besicovitch property} with constant $C$ if every carpet over a finite set $E$ has a subcarpet which also covers $E$ and has multiplicity $\leq C$. We say a sequence of balls $B_{r(i)}(z_i)$ is \textit{incremental} if $r(i)$ is non-increasing and each $z_i \not\in \bigcup_{j < i} B_{r(j)}(z_j)$.

Let $\lbrace Q_i \rbrace_{i=1}^{2^d}$ denote the $2^d$ orthants (the analogue of a quadrant in 2 dimensions) in $\Zd$ - for example the set $\lbrace u \in \Zd :  u_i \geq 0 \, \forall i \rbrace$. The orthants of a rectangular set $B_r(z)$ are given by $B_r(z) \cap (z + Q_i)$ for $1 \leq i \leq 2^d$.

\begin{prop}
Any rectangular metric space $(S,\rho)$ satisfies the Besicovitch property with constant $C = 2^d$.
\end{prop}

\begin{proof}
We may assume $S = \Zd$. Let $\lbrace z_1,...,z_N \rbrace \subseteq \mathbb{Z}^d$ with a carpet $\mathcal{U} = \lbrace B_{r(i)}(z_i) \rbrace_{i=1}^N$. We may reorder so the $r(i)$ are decreasing. We select a subcarpet covering $E$ as follows. Let $I_1 = \lbrace 1 \rbrace$ and, with $I_k$ defined we define by $I_{k+1} = I_k \cup \lbrace m \rbrace$ where 
$$m = \min{\left\lbrace i \in [1,N] : z_i \not\in \bigcup_{j \in I_k} B_{r(j)}(z_j) \right\rbrace}$$
if this exists, otherwise we terminate the process and let our subcarpet $\mathcal{U}' = \lbrace B_{r(i)}(z_i) \rbrace_{i \in I_k}$. Note $\mathcal{U}'$ is an incremental sequence with its natural ordering.

Now, assume for a contradiction that there is $\sigma \in \Zd$ lying in $> 2^d$ elements of $\mathcal{U}'$. Then by pigeonhole principle $\sigma$ must lie in the same orthant, $Q$, of two elements of $\mathcal{U}'$, corresponding to $z_i$ and $z_j$ say. We may assume $i < j$. Then for some numbers $n_l \in \lbrace 0,1 \rbrace$ depending only on $Q$ we have 
\begin{align*}
z_j \in \sigma + \prod_{l=1}^d (-1)^{n_l} [0,\floor{f_l(r_j)}] 
&\subseteq \sigma + \prod_{l=1}^d (-1)^{n_l} [0,\floor{f_l(r_i)}] \\
&\subseteq z_i + \prod_{l=1}^d [-\floor{f_l(r_i)},\floor{f_l(r_i)}] = B_{r(i)}(z_i)
\end{align*}
contradicting the fact that $\mathcal{U}'$ is an incremental sequence.
\end{proof}

It will be useful for us to note the following well known equivalence, a proof can be found in \cite{Hoch1}.

\newpage

\begin{prop} \label{BP=}
Let $S$ be a metric space and $C \in \mathbb{N}$. Then $S$ has the Besicovitch property with constant $C$ if and only if for any carpet $\mathcal{U}$ over a finite set $E$ there is an incremental sequence of sets from $\mathcal{U}$ covering $E$ with multiplicity $\leq C$.
\end{prop}

To see that the maximal inequality holds for $B_n = \lbrace u \in \mathbb{Z}^d : \rho(u,0) \leq n \rbrace$ with $\rho$ a rectangular metric and $C = 4^d$ we refer to a concise proof of the maximal inequality for balls of norms is given in \cite[Inequality 5.3]{Fel1}, attributed to Aaronson and Becker. Upon examining this proof the reader will observe that the same argument, with two changes, goes through for rectangles. The first is that to apply the Besicovitch property in the proof of Inequality 5.2 one needs to intersect with a finite subset, this can be taken arbitrarily large at the end of the proof. The second is that one replaces each occurrence of $B_{2n}$ with $2B_n$, and then applies the modified doubling condition (\ref{mDC}).

\begin{prop}[The Maximal Inequality]
Let $B_1,B_2,...$ be a rectangular summing sequence. Then for any $\phi \in L^1$ and $\epsilon > 0$
$$\mu\left( \sup_{n \geq 1} \left| R_n\phi \right| > \epsilon \right) \leq \frac{4^{d}}{\epsilon} \|\phi\|_1.$$
\end{prop}

\subsection{The Non-singular F{\o}lner Condition}

With the maximal inequality in hand it is sufficient to show that (\ref{nsFC}) holds. We directly adapt the approach in \cite{Hoch1}. The first steps are to prove a pair of disjointification lemmas, the arguments for rectangular metrics are very similar (if not identical) to those for norms and so we only prove the first to illustrate the changes one needs to make. The second step is to prove that $\Zd$ has finite coarse dimension, defined in \cite{Hoch1}, which involves a notion of a thickened boundary of a subset of $\Zd$. It is in the definition of a thickened boundary that our work diverges from that of Hochman, and so here we take care to show that $\Zd$ still has finite coarse dimension with our definition. This choice of definition will mean that the proof of \cite[Theorem 4.4]{Hoch1}, in some sense the central result of the paper, can be copied verbatim. We then mimic Hochman's proof for a variant of the Chacon-Ornstein lemma which implies (\ref{nsFC}).

The author has deliberately kept the statements, definitions and proofs similar to those in \cite{Hoch1} where possible, for ease of comparison.

\subsubsection*{The Disjointification Lemmas}

The first disjointification lemma makes direct use of the Besicovitch property and doubling condition, and so requires some minor but illustrative modifications. Let us first recall some useful terminology. 

We write $\rad{B}$ for the radius of a ball $B$. If $\mathcal{U}$ is a collection of balls we write $\rmin{\mathcal{U}}$ and $\rmax{\mathcal{U}}$ for the minimal and maximal radii of the balls in $\mathcal{U}$. We say that $\mathcal{U}$ is \textit{well-separated} if any two balls in $\mathcal{U}$ are more than $\rmin{\mathcal{U}}$ apart. 

\begin{lem} \label{dis1}
Let $\Zd$ be equipped with a rectangular metric. Then for every finite subset $E$ and every carpet $\mathcal{U}$ over $E$ there is a subcollection $\mathcal{V}$ which covers $E$ and which can be partitioned into $\chi = CD^2 + 1$ ($D = 2^d$) subcollections, each of which is well-separated.
\end{lem}

\begin{proof}
The proof mimics that of \cite[Lemma 3.3]{Hoch1}. We are essentially checking that the balls $B_{\lambda r}$ can be replaced with multiples of rectangles $\lambda B_r$.

Let $z \in \Zd$ and $\mathcal{W}$ be a collection of balls of radius $r$ with centres in $z+3B_r$ and suppose it has multiplicity $\leq C$. Then $\bigcup \mathcal{W} \subseteq z+4B_r$, so
\begin{align*}
|\mathcal{W}||B_r| \leq C |z+4B_r| \leq CD^2 |B_r|
\end{align*}
and hence $|\mathcal{W}| \leq \chi - 1$.

If instead $\mathcal{W}$ contains balls of radius $\geq r$ which all intersect $x+2B_r$, and multiplicity $\leq C$, then we may replace each ball $B$ with a ball of radius $r$ contained in $B$ and centred in $z+3B_r$. We deduce from above that again $|\mathcal{W}| \leq \chi - 1$.

Now, by proposition \ref{BP=} we can find an incremental sequence $\lbrace U_i \rbrace_{i=1}^n \subseteq \mathcal{U}$ covering $E$. We assign colours $1,2,...,\chi$ to the $U_i$ as follows. Colour $U_1$ as you like, and assume we have coloured $U_i$ for $i \leq k$ and consider $U_{k+1}$. Take $r = \rad{U_k}$ and $z$ to be the centre of $U_{k+1}$, by assumption $U_{k+1} \subseteq z + B_r$ and each $U_i$ with $i \leq k$ has radius at least $r$. Therefore, by the above, at most $\chi - 1$ intersect $z + 2B_r$. Give $U_{k+1}$ one of the colours unused by those $U_i$.

Let $\mathcal{V}_k$ be the collection coloured $k$. To see each collection is well-separated note that the points within rectangular distance $r$ of $z + B_r$ are exactly those in $z + 2B_r$, combining this with the colouring process and the fact the radii of the $U_i$ is decreasing gives the result.
\end{proof}

For $S \subseteq \Zd$ let $\chi(S)$ denote the minimal natural number satisfying the conclusion of the proposition, then $\chi(S) \leq \chi(\Zd)$.

\begin{cor} \label{cordis}
Let $S \subset \Zd$ be equipped with rectangular metric, and let $E$, $\mathcal{U}$ and $\chi = \chi(S)$ be as in lemma \ref{dis1}. Assume $\mu$ is a finite measure supported on $E$. Then there is well-separated subset of $\mathcal{U}$ which covers a set of mass $\geq (1/\chi) \mu(E)$.
\end{cor}

Now we move on to the second disjointification lemma. This necessitates a divergence from the definitions in \cite{Hoch1}, where Hochman considers thickened spheres given by the sets $B_{r+t} \setminus B_{r-t}$ for $t \leq r$. In our situation this appears not to be the correct definition. For example, if one considers the case where one side of rectangle is growing exponentially and takes $t = \log{2}$ then for large radii the thickened sphere, which is meant to be a slight thickening of the boundary, would consist of more than half of the points in the rectangle. Instead we take a definition which emulates the behaviour in the case where the metric is given by a norm. 

When $S = \mathbb{Z}^d$ for rectangular balls $B$ let $\partial B$ denote the set of points in $\mathbb{Z}^d$ which lie in the usual topological boundary when considered as a subsets of $\mathbb{R}^d$, and call these sets \textit{boxes}. Another perspective is that the box associated to a rectangle is the collection of points for which some coordinate takes the maximum or minimum value in that coordinate over the rectangle.

For $t \in \mathbb{N}$ we define the \textit{$t$-boundary} $\partial_t B$ to be collection of $z \in \Zd$ which lie within distance $t$ of $\partial B$ with respect to the rectangular metric. Equivalently,
$$\partial_t B = \bigcup_{u \in \partial B} (u + B_t).$$
When $S \subseteq \mathbb{Z}^d$ we take $\partial B$ and $\partial_t B$ to be the intersections of their $\mathbb{Z}^d$ counterparts with $S$. We may refer to a collection of $t$-boundaries, possibly with different values of $t$, as \textit{thick boxes}.

For a collection $\mathcal{U}$ of rectangles we let $\partial \mathcal{U} = \lbrace \partial B : B \in \mathcal{U} \rbrace$. If $\mathcal{U}$ is a collection of boxes we define its radius, and the maximal and minimal radii of $\mathcal{U}$ analogously to rectangles. For $r \in \mathbb{N}$ we say a collection is $r$\textit{-separated} if any two members are more than $r$ away from each other in the rectangular metric. If this is true for $r = \rmin{\mathcal{U}}$ we say the collection is \textit{well-separated}. In particular, if $r > 2t$ and the collection is $r$-separated then the corresponding collection of $t$-boundaries is disjoint. A sequence of carpets $\mathcal{U}_1,...,\mathcal{U}_l$ over a finite set $E$ is called a \textit{stack} and $p$ is its \textit{height}.

Now let us state the second covering lemma.

\begin{lem} \label{2nddis}
Let $S \subseteq \mathbb{Z}^d$ with rectangular metric. Then for $0 < \epsilon, \delta < 1$ and $t \in \mathbb{N}$ let $l = \ceil*{ \frac{2 \chi(S)}{\epsilon \delta} }$ and suppose that
\begin{enumerate}[\normalfont (1)]
\item $\mu$ is a finite measure on $S$.
\item $F \subseteq S$ is finite and $\mu(F) > \delta \mu(S)$.
\item $\mathcal{U}_1,...,\mathcal{U}_p$ is a stack over $F$ with $\rmin{\mathcal{U}_i} \geq \, 2 \, \rmax{\mathcal{U}_{i-1}}$ and $\rmin{\mathcal{U}_1} \geq 2t$.
\item $\mu(\partial_tB) > \epsilon \mu(B)$ for each $B \in \bigcup_i \mathcal{U}_i$.
\end{enumerate}
Then there is some integer $k \geq 2$ and a subcollection $\mathcal{V} \subseteq \bigcup_{i \geq k} \mathcal{U}_i$ of rectangles such that
\begin{enumerate}[\normalfont (i)]
\item $\partial\mathcal{V}$ is well-separated and
\item the set $\bigcup_{B \in \mathcal{V}} \partial_{2r} B$, where $r = \rmax{\mathcal{U}_{k-1}}$, contains more than one half of $F$ with respect to $\mu$.
\end{enumerate}
\end{lem}

Lemma \ref{2nddis} can be proved as with \cite[Lemma 3.3]{Hoch1}. This is because statement of Corollary \ref{cordis} holds unchanged from that paper and the only property of the $t$-boundary of a ball used in the proof is the fact that it contains all points within distance $t$ of the boundary.

\subsubsection*{Coarse Dimension}

Now we shift focus to the second key ingredient of the proof. This is that $\mathbb{Z}^d$ has finite \textit{coarse dimension}, defined as follows, but with respect to norm induced metrics. 

\begin{defn}
For a rectangular metric space $S$ and $R > 1$ the relation $\cdim_{R} S = k$ (read: $S$ has coarse dimension $k$ at scales $\geq R$) is defined by recursion on $k$ by:
\begin{enumerate}[\normalfont (i)]
\item $\cdim_{R} S = -1$ for $S = \emptyset$ and any $R$,
\item $\cdim_{R} S$ is the minimum integer $k$ for which $\cdim_{tR} \partial B_r(s) \leq k - 1$ for any $t \geq 1$, $r \geq t R$ and $s \in S$.
\end{enumerate}
\end{defn}

As such, this is where we depart further from \cite{Hoch1} and make direct use of the properties of rectangular metrics. We use the same definitionof coarse dimension, except with $t$-boundary as defined for rectangles.

The following proposition will be useful in the proof that $\mathbb{Z}^d$ has finite coarse dimension with respect to the redefined boundary. 

For $e \in \lbrace \pm e_i : 1 \leq i \leq d \rbrace$ let $F_{r,u}(e)$ be the face of $B_r(u)$ in direction $e$ from $u$, i.e. those points in $B_r(u)$ whose projection onto $e$ is maximal. The \textit{face} of the thickened boundary $\partial_t B_r(u)$ in direction $e$ is the set of points within distance $t$ of $F_{r,u}(e)$ and is denoted by $\partial_t F_{r,u}(e)$.

\begin{prop}
Let $(\mathbb{Z}^d,\rho)$ be a rectangular metric space. Then there are $R = R(\rho) > 1$ and $k \in \mathbb{N}$ with the following property: given $z_1,...,z_k \in \mathbb{Z}^d$, $t(1),...,t(k) \geq 1$ and a decreasing sequence $r(1),...,r(k)$ with $r(k) \geq t(1)...t(k) R$ such that $z_i \in \bigcap_{j < i} \partial_{t(j)} B_{r(j)}(z_j)$ then $$\bigcap_{i = 1}^k \partial_{t(i)} B_{r(i)}(z_i) = \emptyset.$$
\end{prop}

\begin{proof}
For notational clarity we write $r_i = r(i)$ and $t_i = t(i)$ in this proof.

Fix $R > 5n$ with $n \in \mathbb{N}$ chosen large enough for $nf_i(1) \geq 1$ for all $i \in [1,d]$. We use induction on the $d$ to prove that there is $k = k(d)$ with the required property. 

For $d = 1$ let $k = 2$. Let $f = f_1$. The set $\partial_{t(1)} B_{r(1)}(z_1)$ is a union of two closed intervals length $2\floor{f(t_1)}+1$ centred on $\pm \floor{f(r_1)}$. These intervals are disjoint as $r(1) > t(1)$. We may assume $z_2$ lies in the lower interval. Now since $R > 5n$ we have 
$$\floor{f(r_2)} > f(r_2) - 1 \geq f(2t_1+t_2+2n)-1 \geq 2\floor{f(t_1)} + \floor{f(t_2)}+1$$
using superadditivity of $f$ and the choice of $n$. In particular $\partial_{t(2)} B_{r(2)}(x_2)$ does not intersect the lower interval.

Also,
$$\floor{f(r_2)} + \floor{f(t_2)} < 2(\floor{f(r_1)} - \floor{f(t_1)})$$
else using $R > 5n$ the fact the $r(i)$ are decreasing
\begin{align*}
2\floor{f(t_1)} + \floor{f(t_2)} 
&\geq	 2\floor{f(r_1)} - \floor{f(r_2)} \\
&\geq 	 \floor{f(r_1)}
\geq 	 f(2t_1 + t_2 + 2n) - 1
>		 2\floor{f(t_1)} + \floor{f(t_2)}.
\end{align*}
This means that $\partial_{t(2)} B_{r(2)}(z_2)$ also does not intersect the upper interval, and the claim follows.

Now, assume we have proved $k(d-1)$ exists. Suppose $k \geq 2dk(d-1)+2$. By the pigeonhole principle the thickening some face $F(e)$ of $B_{r(1)}(z_1)$ contains $k(d-1)+1$ of the points $z_2,...,z_{k(d)}$. As these are the only points used from here we may assume they are $z_2,...,z_{k(d-1)+2}$. Using essentially the argument from the initial step the thickened faces in directions $\pm e$ of each $\lbrace \partial_{t(i)} B_{r(i)}(z_i) \rbrace_{i = 2}^{2k(d-1)+2}$ cannot intersect the thickened faces $F(\pm e)$ of $\partial_{t(1)} B_{r(1)}(z_1)$. Therefore the $\partial_{t(i)}B_{r(i)}(z_i)$ intersect in $\partial_t F(e)$ only if the projections of $\partial_{t(i)}B_{r(i)}(z_i) \cap \partial_{t(1)}F(e)$ along $e$ onto $F(e)$ intersect. These projections are exactly thick boxes for projection of our rectangular metric in direction $e$, so we may apply the previous case to deduce that 
$$\partial_{t(1)}F(e) \cap \bigcap_{i=2}^{k(d-1)+1} \partial_{t(i)}B_{r(i)}(z_i) = \emptyset$$
but by assumption $z_{k(d-1)+2}$ lies in that intersection. Hence $k < 2dk(d-1)+2$ and so $k(d) \leq 2dk(d-1)+1$.
\end{proof}

Using the above we are able to prove the claim.

\begin{prop}
$\mathbb{Z}^d$ has finite coarse dimension with respect to any rectangular metric.
\end{prop}

\begin{proof}
As before, we write $r_i = r(i)$ and $t_i = t(i)$ in this proof.

Let $R = R(\rho)$ and $k' = k$ from the previous proposition. Let $k'' \in \mathbb{N}$, to be determined, and $k = k'k''+1$. In order to show $\mathbb{Z}^d$ has finite coarse dimension is suffices to show that if we are given
\begin{enumerate}
\item $t(1),...,t(k) \geq 1$,
\item $r(1),...,r(k)$ such that $r(i) \geq t(1)...t(k)R$ and
\item points $z_1,...,z_k \in \mathbb{Z}^d$ such that $z_i \in \bigcap_{j < i} \partial_{t(j)} B_{r(j)}(z_j)$ for $j < i$,
\end{enumerate}
then $\bigcap_{i=1}^k \partial_{t(i)} B_{r(i)}(z_i) = \emptyset$.

By the previous proposition it suffices to find a subsequence length $k'$ for which the radii are decreasing. Consider the points $z_2,...,z_{l}$ ($l \geq 2$) and suppose $r(j) > r(1)$ for each $2 \leq j \leq l$. Each of these points lies inside $\partial_{t(1)} B_{r(1)}(z_1)$, by assumption. Moreover if $i > j$ then 
\begin{align*}
z_j 
&\not\in	z_i + \prod_{m=1}^d (-\floor{f_m(r_i)}+\floor{f_m(t_i)},\floor{f_m(r_i)}-				\floor{f_m(t_i)}) \\
&\supseteq	z_i + \prod_{m=1}^d (-\floor{f_m(r_1)}+\floor{f_m(r_1/R)},\floor{f_m(r_1)}-\floor{f_m(r_1/R)}).
\end{align*}
Let $A = \prod_{m=1}^d (-\floor{f_m(r_1)}+\floor{f_m(r_1/R)},\floor{f_m(r_1)}-				\floor{f_m(r_1/R)})$. The final line implies that we also have 
$z_i \not\in z_j + A$. Now, $z_2,...,z_l$ is a collection of points contained by ${B = \partial_{t(1)} B_{r(1)}(z_1) \cup B_{r(1)}(z_1)}$ such that $z_i \not \in z_j + A$ for all $i \neq j$. Then the sets $z_j + \frac{1}{2}A$ are disjoint and each $B \cap (z_j + \frac{1}{2}A)$ contains at least one orthant of $z_j + \frac{1}{2}A$, and hence at least
$$\prod_{m=1}^d \left\lfloor \frac{1}{2}(\floor{f_m(r_1)}-	\floor{f_m(r_1/R)}-1) \right\rfloor $$
points. By the disjointness we must have
\begin{align*}
(l-1) \prod_{a=1}^d \left\lfloor \frac{1}{2}(\floor{f_m(r_1)}-	\floor{f_m(r_1/R)}-1) \right\rfloor \leq \prod_{m=1}^d \left( 2(\floor{f_m(r_1)} + \floor{f_m(t_1)}) + 1 \right)
\end{align*}
i.e.
\begin{align*}
l \leq 1 + 2^d \prod_{m=1}^d \frac{ 2(\floor{f_m(r_1)} + \floor{f_m(r_1/R)}) + 1 }{\floor{f_m(r_1)}-	\floor{f_m(r_1/R)} - 3}
\end{align*}
Dividing through each fraction by $\floor{f_m(r_1)}$ and recalling that $f_m(r_1) \geq f_m(5n) \geq 5$ and 
$$\frac{\floor{f_m(r_1/R)}}{\floor{f_m(r_1)}} \leq \frac{\floor{f_m(r_1/5)}}{5\floor{f_m(r_1/5)}-1} \leq \frac{1}{4}$$
so
\begin{align*}
l \leq 1 + 2^d \prod_{m=1}^d \frac{ 2(1 + 1/4) + 1/5 }{1-	1/4 - 3/5} \leq 36^d +1
\end{align*}
Therefore if we take $k'' > 36^d + 1$ then some $r(j) \leq r(1)$ for $2 \leq j \leq k''$. We can then repeat this process with $r(j)$ and so on to find a subsequence with decreasing radii satisfying the conditions, which will have length at least $k'$ by our choice of $k$.
\end{proof}

It should be clear that if $\Zd$ has finite coarse dimension at scales $R$ then any subset will have coarse dimension at most $\cdim_{R} \Zd$ also.

These results can be used to prove a rectangular analogue of Hochman's main theorem.

\begin{theorem} \label{techres}
Let $S \subseteq \Zd$ with rectangular metric, fix $t, \chi,k \in \mathbb{N}$ and $0 < \epsilon, \delta < 1$. Set $q = 1000^{k^2}\left(\frac{200 \chi^2}{\epsilon^2 \delta^3}\right)^k$. Suppose that
\begin{enumerate}[\normalfont (1)]
\item $\chi(S) \leq \chi$ and $\cdim_{R} S = k$ for some $R > 2$,
\item $\mu$ is a finite measure on $S$,
\item $E \subseteq S$ is finite,
\item $\mathcal{U}_1,...,\mathcal{U}_q$ is a stack over $F$ with
\begin{enumerate}
\item $\rmin{\mathcal{U}}_i > (\rmax{\mathcal{U}_{i-1}})^2$,
\item $\rmin{\mathcal{U}}_1 > \max{(2t,R)}$,
\end{enumerate}
\item $\mu(\partial_tB) > \epsilon \mu(B)$ for each $B \in \bigcup_i \mathcal{U}_i$.
\end{enumerate}
Then $\mu(E) \leq \delta \mu(S)$.
\end{theorem}

The proof is also the same as given in \cite{Hoch1}, as the only property of the thickening used directly in the proof is that it contains all the points within a certain distance from the boundary and the previous results cover the rest of the argument.

\subsubsection*{The Non-singular F{\o}lner Condition}

The above theorem is used to prove the Chacon-Ornstein type result, below.

\begin{theorem}
Suppose we have a non-singular $\mathbb{Z}^d$-action on a probability space $(X,\mu)$ and $B_n = \lbrace u \in \mathbb{Z}^d : \rho(u,0) \leq n \rbrace$ for some rectangular metric $\rho$ on $\mathbb{Z}^d$. Then for any $t \in \mathbb{N}$
\begin{align*}
\frac{\sum_{u \in \partial_t B_n} \omega_u }{\sum_{u \in B_n} \omega_u} \rightarrow 0 \quad \text{a.s..}
\end{align*}
\end{theorem}

\begin{proof}
Let $k = \cdim_{R} \mathbb{Z}^d$, $\chi = \chi(\mathbb{Z}^d)$ and $R = R(\rho)$. Write
$$p_n(x) = \frac{\sum_{u \in \partial_t B_n} \omega_u(x) }{\sum_{u \in B_n} \omega_u(x)}$$
and set
$$A_\epsilon = \lbrace x \in X : \limsup p_n(x) > \epsilon \rbrace.$$
It suffices to show $\mu(A_\epsilon) = 0$ for all $\epsilon > 0$, so let us assume for a contradiction $\mu(A_\epsilon) > 0$.

As in \cite{Hoch1} we construct a sequence of natural numbers such that
$$r_1^- \leq r_1^+ \leq r_2^- \leq r_2^+ \leq r_1^- \leq r_1^+ \leq ...$$
such that $r_1^- > \max(2t,R)$ and $r_i^- > (r_{i-1}^+)^2$, and a set of points $A \subseteq A_\epsilon$ so that for every $x \in A$ and $i \geq 1$ there is an $n_i = n_i(x) \in [r_i^-,r_i^+]$ with $p(n_i,x) > \epsilon$ and $\mu(A) \geq \frac{1}{2}\mu(A_\epsilon)$.

We are now ready to apply Hochman's main theorem, via a transference principle. Fix $0< \delta <1$, $n > r_q^+ + t$ and $q = q(k,\chi,\epsilon,\delta)$ as in the previous theorem. Then
$$\mu(A) = \frac{1}{|B_n|}  \int \sum_{u \in B_n} 1_A(ux) \omega_u(x) d\mu(x),$$
and we bound the integrand. Fix $x \in X$ and define a measure $\nu = \nu_{x,n}$ on $2B_n$ by $\nu(\lbrace u \rbrace) = \omega_u(x)$. Let 
$$V = V_{x,n} = \lbrace v \in B_n : v \cdot x \in A \rbrace.$$
Then for each $1 \leq i \leq q$ and $v \in V$ let $m_i(v) \in [r_i^-,r_i^+]$ such that $p_{m_i(v)}(vx) > \epsilon$, or equivalently 
$$\sum_{u \in \partial_t (v+B_{m_i(v)})} \omega_u(x) > \epsilon \sum_{u \in v+B_{m_i(v)}} \omega_u(x).$$
Since 
$$\partial_t (v+B_{m_i(v)}) = v+\partial_t (B_{m_i(v)}) \subseteq B_n + B_{r_q^+ +t} \subseteq 2B_n$$
this means 
$$\nu(\partial_t (v+B_{m_i(v)})) > \epsilon \nu(v+B_{m_i(v)}).$$
It follows that the carpets $\mathcal{U}_i = \lbrace v+B_{m_i(v)} : v \in V \rbrace$ form a stack of height $q$ over $V$ satisfying the conditions of the theorem \ref{techres} (with $X = 2B_n$). Hence
$$\sum_{u \in B_n} 1_A(ux)\omega_u(x) = \nu(V) \leq \delta \nu(2B_n) = \delta \sum_{u \in 2B_n} \omega_u(x).$$
and so
$$\mu(A) \leq \frac{\delta}{|B_n|} \int \sum_{u \in 2B_n} \omega_u(x) d\mu(x) = \delta\frac{|2B_n|}{|B_n|} \leq 2^d \delta.$$
Since $\delta$ can be made arbitrarily small we see that $\mu(A) =0$ forcing $\mu(A_\epsilon) =0$ for a contradiction.
\end{proof}

For fixed $v \in \Zd$ fixing $t \geq \rho(0,v)$ ensures that $\partial_t B_n \supseteq B_n \triangle (v + B_n)$ and hence that (nsFC) holds. Putting these results together completes the proof of theorem \ref{recET}. The techniques of Mortiss \cite{Mort1} combined with \ref{recET} give the following corollary, which is used to prove the main result of section \ref{prodact}.

\section{Critical Dimension for $\Zd$-actions} \label{prodact}

We now have a varied collection of summing sequences in $\Zd$ for which the ergodic theorem holds, and hence for which the critical dimensions are invariants of metric isomorphism. In this section we restrict attention to these sequences in order to address the first question raised in the introduction: how do $\alpha$ and $\beta$ depend on the choice of summing sequence?

\subsection{Critical Dimension for Balls of Norms}

We begin by showing that the critical dimensions for balls of a norm are independent of the choice of norm. Since every norm on $\mathbb{R}^d$ is equivalent their balls $B_n = nB$, where $B$ is the unit ball, grow at the same rate. In addition, for some $k \in \mathbb{N}$ we have $k^{-1} B \subseteq B' \subseteq kB$. The latter ensures that the sequences $nB$ and $nB'$ are, in some sense, intertwined. We prove that the critical dimensions for any pair of sequences with these two properties are equivalent. 

The techniques used here hold for a countable group $G$, as in the introduction, so we temporarily return to that setting.

Let each of $\lbrace A_n \rbrace_{n=1}^\infty$ and $\lbrace A_n' \rbrace_{n=1}^\infty$ be an increasing sequence of subsets of $G$. We shall say these sequences are \emph{interweaving} if for all $n \in \mathbb{N}$ there exists $N,N' \in \mathbb{N}$ such that $A_n \subseteq A_{N}'$ and $A_n' \subseteq A_{N'}$.

Given two interweaving sequences $\lbrace A_n \rbrace_{n=1}^\infty$ and $\lbrace A_n' \rbrace_{n=1}^\infty$ let 
\begin{align*}
m(n) = \max \left( k \geq 0 : A_k' \subseteq A_n \right) & & \text{and} & & m'(n) = \max \left( k \geq 0 : A_k \subseteq A_n' \right)
\end{align*}
where for technical reasons we take $A_0 = \emptyset = A_0'$. Then both $m(n)$ and $m'(n)$ are increasing with $n$ and diverge as $n \rightarrow \infty$. We say these sequences have \emph{comparable growth} if there exists $C \in (0,1)$ such that for all $n$ sufficiently large
\begin{align*}
C \leq \frac{\big| A_{m(n)}' \big|}{|A_n|} \leq C^{-1} & & \text{and} & & C \leq \frac{\big| A_{m'(n)} \big|}{|A_n'|} \leq C^{-1}.
\end{align*}

\begin{prop} \label{=CD}
Let $G$ be a countable group with a non-singular ergodic action on a standard finite measure space $(X,\mu)$. Let $\lbrace A_n\rbrace_{n==1}^\infty$ and $\lbrace A_n'\rbrace_{n=1}^\infty$ be a pair of interweaving sequences in $G$ of comparable growth. Then $L_t = L_t'$ and $U_t = U_t'$. In particular the two sequences give the same upper and lower critical dimensions.
\end{prop}

\begin{proof}
We just tackle the lower case as the upper case is similar. Observe that for all $n \geq N$ 
\begin{align*}
\frac{1}{|A_n|^t} \sum_{g \in A_n} \omega_g(x) 
\geq \left( \frac{\big| A_{m(n)}' \big|}{|A_n'|} \right)^t \frac{1}{\big| A_{m(n)}' \big|^t} \sum_{g \in A_{m(n)}'} \omega_g(x) 
\geq C^{|t|} \frac{1}{\big| A_{m(n)}' \big|^t} \sum_{g \in A_{m(n)}'} \omega_g(x) 
\end{align*}
and hence
\begin{align*}
\inf_{n \geq N} \frac{1}{|A_n|^t} \sum_{g \in A_n} \omega_g(x) 
\geq  C^{|t|} \inf_{n \geq N} \frac{1}{\big| A_{m(n)}' \big|^t} \sum_{g \in A_{m(n)}'} \omega_g(x)  \geq  C^{|t|} \inf_{n \geq m(N)} \frac{1}{\big| A_{n}' \big|^t} \sum_{g \in A_{n}'} \omega_g(x).
\end{align*}
By letting $N \rightarrow \infty$, and recalling that $m(N) \rightarrow \infty$ as $n \rightarrow \infty$ it follows that
\begin{align*}
\liminf_{n \rightarrow \infty} \frac{1}{|A_n|^t} \sum_{g \in A_n} \omega_g(x)
\geq  C^{|t|} \liminf_{n \rightarrow \infty} \frac{1}{\left| A_{n}' \right|^t} \sum_{g \in A_{n}'} \omega_g(x)
\end{align*}
and hence $L_t' \subseteq L_t$. The same argument holds with the sequences exchanged, so the claim follows.
\end{proof}

Returning to $G = \Zd$ let $B$ be the unit ball of the supremum norm and $B'$ of some other norm. Since $k^{-1} B \subseteq B' \subseteq kB$ for some $k$ these sequences are interweaving. To see that the sequences $nB$ and $nB'$ have comparable growth observe, for example, that
$$\frac{| \lfloor nk^{-1} \rfloor B|}{|nB'|} \geq \frac{| \lfloor nk^{-1} \rfloor B|}{|nkB|} = \left( \frac{2 \lfloor nk^{-1} \rfloor +1}{2nk+1} \right)^d \geq \left( \frac{  2nk^{-1} - 1}{2nk +1} \right)^d \rightarrow \frac{1}{k^{2d}}.$$
Of course a similar argument holds with the places of $B'$ and $B$ exchanged. Proposition \ref{=CD} then shows every sequences of balls of norms produces the same critical dimension.

As one might expect it is not too difficult to see that the comparable growth rates are necessary to the above argument. Consider, for example, the sequences $A_n' = [-n,n]^2$ and $A_n = [-\floor{e^n-1},\floor{e^n-1}] \times [-n,n]$ in $\mathbb{Z}^2$. We have $m(n) = n$ and hence 
$$\frac{\big| A_{m(n)}' \big|}{|A_n|} = \frac{(2n+1)^2}{(2n+1)(2\floor{e^n-1}+1)} \rightarrow 0.$$
This means that the argument used in the above proof fails if one attempts to compare balls of arbitrary rectangular metrics to those of norms. Next we show that these sequences give rise to different critical dimensions for numerous actions. 

\subsection{Critical Dimension for Product Measure Spaces}

We examine non-singular product actions, which are constructed as follows. Suppose that for each $1 \leq i \leq d$ we are given a non-singular transformation $T_i : X_i \rightarrow X_i$ on a probability space $(X_i,\mu_i)$, the factors of the product. We can define a non-singular $\mathbb{Z}^d$-action on the product measure space $X = X_1 \times ... \times X_d$ with measure $\mu = \mu_1 \otimes ...\otimes \mu_d$ via 
$$(u_1,...,u_d) \cdot (x_1,...,x_n) = (T_1^{u_1}x_1,...,T_d^{u_d}x_d).$$ 
This action is ergodic if and only if every $T_i$ is ergodic.

We consider the upper and lower critical dimensions with respect to sequences of rectangles $B_n = B_n^1 \times ... \times B_n^d$ where each $B_n^i = [-s_i(n),s_i(n)]$ for some increasing functions $s_i : \mathbb{N}_0 \rightarrow \mathbb{N}_0$. This setup includes rectangular summing sequences. For each $1 \leq i \leq d$ we write $\alpha_i$ and $\beta_i$ for the lower and upper critical dimensions of $T_i$ with respect to $[-n,n]$, taken in the space $(X_i,\mu_i)$. 

Given two increasing functions $s,s': \mathbb{N} \rightarrow \mathbb{N}_{>1}$ we write $s \lesssim s'$ and say $s$ is \textit{controlled by} $s'$ if 
$$\liminf_{n \rightarrow \infty} \frac{\log{s'(n)}}{\log{s(n)}} > 0.$$ $\lesssim$ defines a preorder on the space such functions, and this preorder is total. We can use $\lesssim$ to define an equivalence relation by declaring that $s$ and $s'$ have \textit{equivalent growth}, denoted $s \approx s'$, if both $s \lesssim s'$ and $s \lesssim s'$, i.e. if
$$0 < \liminf_{n \rightarrow \infty} \frac{\log{s'(n)}}{\log{s(n)}} \leq \limsup_{n \rightarrow \infty} \frac{\log{s'(n)}}{\log{s(n)}} < \infty.$$
This definition ensures that each function $\floor{n^t}$ for $t > 0$ is in the same equivalence class, but $\floor{e^n -1}$ is strictly greater.
 
Using the axiom of choice we may fix a representative of each equivalence class. Suppose that $\bar{s}$ is the representative of the equivalence class of $s$, then we set
\begin{align*}
a(s) = \liminf_{n \rightarrow \infty} \frac{\log{s(n)}}{\log{\bar{s}(n)}} & & \text{ and } & & b(s) = \limsup_{n \rightarrow \infty} \frac{\log{s(n)}}{\log{\bar{s}(n)}}.
\end{align*}
When referring to rectangles $B_n$ as above let us write $a_i = a(s_i)$ and $b_i = b(s_i)$ wherever there is no ambiguity.

Our first, foundational, result of this part provides bounds for the critical dimensions with respect to the rectangles $B_n$ in terms of the critical dimensions of the product transformations and the growth rates of the rectangle sides.

\begin{prop} \label{prodCD}
Let $\Zd$ act on a product space $(X,\mu)$ via a non-singular and ergodic product action, as described above. Let $D \subseteq [1,d]$ such that for each $i \in D$ the function $s_i$ is a greatest element in $\lbrace s_1,...,s_d \rbrace$ with respect to $\lesssim$. Then
\begin{align*}
\frac{\sum_{i \in D} a_i \alpha_i}{\sum_{i \in D} b_i} \leq \alpha(B_n) \leq \beta(B_n) \leq \frac{\sum_{i \in D} b_i \beta_i}{\sum_{i \in D} a_i} .
\end{align*} 
\end{prop}

The inner bound is true by definition, the two outer bounds have slightly different proofs but both rely on two key ideas. 

The first is that a small portion of the growth from the fastest growing sides can be used to dominate and hence neglect the behaviour from the slower growing sides. The second idea is that the rates of growth from the fastest growing sides can be compared using the representative of their equivalence class, resulting in the weighted average of critical dimensions seen above.

We first prove the lower bound, where growth from the slow growing sides is absorbed by the faster sides.

\begin{lem}
Let $\Zd$ act on a product space $X$ via a non-singular and ergodic product action, as described above. Let $D \subseteq [1,d]$ such that for each $i \in D$ the function $s_i$ is a greatest element in $\lbrace s_1,...,s_d \rbrace$ with respect to $\lesssim$. Then
\begin{align*}
\alpha(B_n) \geq \frac{\sum_{i \in D} a_i \alpha_i}{\sum_{i \in D} b_i}.
\end{align*}
\end{lem}

\begin{proof}
Suppose 
$$t = \frac{\sum_{i \in D} (a_i - \epsilon) (\alpha_i - 2\epsilon)}{\sum_{i \in D} b_i}$$ 
for some $\epsilon > 0$. It follows from considering cylinder sets and applying Fubini's theorem that for $u \in \mathbb{Z}^d$ we have $\omega_u(x) = \prod_{i=1}^d \omega_{u_i}^i(x)$ where 
$$\omega_j^i(x) = \frac{d \mu_i \circ T_i^j}{d \mu_i}(x_i).$$
Then
\begin{align} \label{sumform}
\frac{1}{|B_n|^{t}} \sum_{u \in B_n} \omega_u = \frac{1}{2^{dt}} \frac{1}{(\prod_{i=1}^d s_i(n))^t} \prod_{i=1}^d  \sum_{j \in B_n^i} \omega^i_j.
\end{align}

Let $\bar{s}$ be the representative of the growth equivalence class of the $s_i$ with $i \in D$ and fix a positive real number $\delta$. For $i \not\in D$ we have 
$$\liminf_{n \rightarrow \infty} \frac{\log{s_i(n)}}{\log{\bar{s}(n)}} = 0.$$ 
Hence for $i \not\in D$ for all $n$ sufficiently large $s_i(n) \leq \bar{s}(n)^{\delta}$. By definition for $i \in D$ for large $n$ we must have $\bar{s}(n)^{a_i - \epsilon} \leq s(n) \leq \bar{s}(n)^{b_i + \delta}$. Therefore, for all sufficiently large $n$ we have 
$$\prod_{i=1}^d s_i(n) \leq (\bar{s}(n))^{d\delta + \sum_{i \in D} b_i}$$
and so for some $\eta = O(\delta)$ we have
$$\left( \prod_{i=1}^d s_i(n) \right)^t \leq (\bar{s}(n))^{\sum_{i \in D} (a_i - \epsilon) (\alpha_i + \eta - 2\epsilon)} \leq \prod_{i \in D} (s_i(n))^{\alpha_i + \eta - 2\epsilon}.$$

As we retain the freedom to shrink $\delta$ we can assume that each $\eta < \epsilon$ to deduce that for large enough $n$
\begin{align*} \label{lbound}
\frac{1}{|B_n|^{t}} \sum_{u \in B_n} \omega_u \geq \frac{1}{2^{dt}} \left(\prod_{i \not\in D}  \sum_{j \in B_n^i} \omega^i_j \right) \left( \prod_{i \in D} \frac{1}{s_i(n)^{\alpha_i - \epsilon}}  \sum_{j \in B_n^i} \omega^i_j \right).
\end{align*}
The first bracket is always at least $1$ and each term of the latter product diverges to infinity. Hence we see that $\alpha \geq t$, but since $\epsilon > 0$ was arbitrary the inequality follows.
\end{proof}

For the upper bound a little of the growth from the fast growing sides is used to dominate the slower sides.

\begin{lem}
Let $\Zd$ act on a product space $X$ via a non-singular and ergodic product action, as described above. Let $D \subseteq [1,d]$ such that for each $i \in D$ the function $s_i$ is a greatest element in $\lbrace s_1,...,s_d \rbrace$ with respect to $\lesssim$. Then
\begin{align*}
\beta(B_n) \leq \frac{\sum_{i \in D} b_i \beta_i}{\sum_{i \in D} a_i}.
\end{align*}
\end{lem}

\begin{proof}
The result is trivial if any $\beta_i = \infty$, so assume not. Suppose 
$$t = \frac{\sum_{i \in D} (b_i + \epsilon) (\beta_i + 2\epsilon)}{\sum_{i \in D} a_i}$$ 
for some $\epsilon > 0$. Let $\bar{s}$ be the representative of the $s_i$ with $i \in D$ and fix $\delta > 0$. By definition for $i \in D$ and $n$ sufficiently large $\bar{s}(n)^{a_i - \delta} \leq s(n) \leq \bar{s}(n)^{b_i+\epsilon}$. Hence for these $n$
$$\prod_{i=1}^d s_i(n) \geq \bar{s}(n)^{- |D|\delta + \sum_{i \in D} a_i}$$
and so for some $\eta = O(\delta)$ we have
$$\left( \prod_{i=1}^d s_i(n) \right)^t 
\geq \bar{s}(n)^{- \eta + \sum_{i \in D} (b_i + \epsilon) (\beta_i + 2\epsilon)} 
\geq \bar{s}(n)^{- \eta + \epsilon\sum_{i \in D} b_i } \left(\prod_{i \in D} s_i(n)^{\beta_i + \epsilon}\right).$$
By shrinking $\delta$ we can assume that $c = \frac{1}{d-|D|} \left( \epsilon\sum_{i \in D} b_i - \eta \right) > 0$ and use (\ref{sumform}) to deduce that for large $n$
\begin{align*} 
\frac{1}{|B_n|^{t}} \sum_{u \in B_n} \omega_u \leq \frac{1}{2^{dt}} \left(\prod_{i \not\in D} \frac{1}{\bar{s}(n)^c}  \sum_{j \in B_n^i} \omega^i_j \right) \left( \prod_{i \in D} \frac{1}{s_i(n)^{\beta_i + \epsilon}}  \sum_{j \in B_n^i} \omega^i_j \right).
\end{align*}
For each $i \not\in D$ eventually $\bar{s}(n)^c \geq s_i(n)^{\beta_i+\delta}$ and so each term in the first product tends to $0$. Similarly with each of the terms in the second product. Hence we see that $\beta < t$, but since $\epsilon > 0$ was arbitrary the inequality follows.
\end{proof}

This completes the proof of proposition \ref{prodCD}. We can combine it with the integer theory to start to answer the earlier question about dependence on the summing sequence. The integer theory predominantly sums over the sets $[1,n]$ so it will be useful to examine what the critical dimension of a $\mathbb{Z}$-action with respect to $[1,n]$ says about the critical dimension with respect to $[-n,n]$.

Let $T : X \rightarrow X$ be a non-singular transformation describing a $\mathbb{Z}$-action. We shall refer to the critical dimensions of $T$ with the summing sets $[1,n]$ as \textit{standard} and denote the lower and upper standard critical dimensions by $\alpha_+$ and $\beta_+$ respectively. We will denote the lower and upper standard critical dimensions of $T^{-1}$ by $\alpha_-$ and $\beta_-$. Let $L_t^+$, $L_t^-$ denote $L_t$ for $T$ and $T^{-1}$ respectively, with the standard summing sets, and similarly with $U_t$.

\begin{lem} \label{stansym}
Let $T : X \rightarrow X$ determine a non-singular $\mathbb{Z}$-action. Let $\alpha$ and $\beta$ be the critical dimensions with respect to $[-n,n]$. Then $\max(\alpha_+,\alpha_-) \leq \alpha \leq \beta \leq \max( \beta_+, \beta_-)$.
\end{lem}

\begin{proof}
We first prove the result for the lower critical dimension. Observe that
\begin{align*}
\liminf_{n \rightarrow \infty} \frac{1}{(2n+1)^t} \sum_{i = -n}^n \omega_i(x) 
&= \frac{1}{2^t} \liminf_{n \rightarrow \infty} \frac{1}{n^t} \sum_{i = -n}^{-1} \omega_i(x) + \frac{1}{n^t} \sum_{i = 1}^{n} \omega_i(x) \\
&\geq \frac{1}{2^t} \liminf_{n \rightarrow \infty} \frac{1}{n^t} \sum_{i = -n}^{-1} \omega_i(x) +\frac{1}{2^t} \liminf_{n \rightarrow \infty}  \frac{1}{n^t} \sum_{i = 1}^{n} \omega_i(x).
\end{align*}
Hence $L_t \supseteq L_t^+ \cup L_t^-$ and the result follows. In the other case we get
\begin{align*}
\limsup_{n \rightarrow \infty} \frac{1}{(2n+1)^t} \sum_{i = -n}^n \omega_i(x) 
&\leq \frac{1}{2^t} \limsup_{n \rightarrow \infty} \frac{1}{n^t} \sum_{i = -n}^{-1} \omega_i(x) +\frac{1}{2^t} \limsup_{n \rightarrow \infty}  \frac{1}{n^t} \sum_{i = 1}^{n} \omega_i(x).
\end{align*}
Therefore $U_t \supseteq U_t^+ \cap U_t^-$ and we are done.
\end{proof}

In particular, if the standard upper and lower critical dimensions of $T$ agree and those of $T^{-1}$ do also then $\alpha = \max(\alpha_+,\alpha_-) = \beta$. The following theorem of Mortiss and Dooley provides a number of situations where this is the case. 

\begin{theorem} [see \cite{DoolMort1}]
Let $T$ denote the odometer transformation on $(\prod_{i=1}^\infty \mathbb{Z}_2, \bigotimes_{i=1}^\infty \mu_i)$. Then the lower and upper critical dimensions are given by
\begin{align*}
\alpha = \liminf_{n \rightarrow \infty} - \frac{1}{n} \sum_{i=1}^n \log_2{\mu_i(x_i)} = \liminf_{n \rightarrow \infty} \frac{1}{n} \sum_{i=1}^n H(\mu_i) 
\end{align*}
and
\begin{align*}
\beta = \limsup_{n \rightarrow \infty} - \frac{1}{n} \sum_{i=1}^n \log_2{\mu_i(x_i)} = \limsup_{n \rightarrow \infty} \frac{1}{n} \sum_{i=1}^n H(\mu_i) 
\end{align*}
for a.e. $x \in X$, where $H(\mu_i) = - \sum_{j=0}^1 \mu_i(j) \log_2(\mu_i(j))$, the entropy of the measure $\mu_i$.
\end{theorem}

The entropy $H(\mu)$ of the measure $\mu$ on $\lbrace 0,1 \rbrace$ can be chosen to take any value between $0$ and $1$, by varying $p \in (0,1)$ where $\mu(0) = p$. It is clear that for many choices of product measure $\bigotimes_{i=1}^\infty \mu_i$ the sequence $\frac{1}{n}\sum_{i=1}^n H(\mu_i)$ converges as $n \rightarrow \infty$. In this case the upper and lower critical dimensions are equal. Moreover any value in $(0,1)$ can be achieved by the dimensions.

Another consequence theorem is that for an odometer action $T$ on $(\prod_{i=1}^\infty \mathbb{Z}_2, \bigotimes_{i=1}^\infty \mu_i)$ the inverse $T^{-1}$ has the same upper and lower critical dimensions as $T$. This follows from how $T^{-1}$ can also be considered as an odometer on the same space, with the roles of $0$ and $1$ reversed, and the fact that $H(\mu_i) = H(\nu_i)$ where $\nu_i(0) = 1 - \mu_i(0)$.

These observations combined with lemma \ref{stansym} ensure we can produce examples of transformations with a single critical dimension $\alpha = \beta = \gamma$ with respect to $[-n,n]$ for any $\gamma \in (0,1)$. 

If we input these $T_i$ into proposition \ref{prodCD}, and choose the $s_i$ to ensure $a_i = b_i$ for all $i \in D$, then the resulting actions will have critical dimension
\begin{align*}
\gamma(B_n)  = \frac{\sum_{i \in D} a_i \gamma_i}{\sum_{i \in D} a_i}.
\end{align*}
We are now equipped to examine some specific examples which answer some of our earlier questions.

\subsubsection*{Values taken by the Critical Dimension}

The simplest examples to consider are those where $s_1(n) = s_2(n) = ... = n$ which all satisfy $a(s_i) = 1$ with respect the natural choice of representative of their class, $\bar{s}(n) = n$. Then in the above circumstances there is a single critical dimension
$$\gamma = \frac{\gamma_1 + ... + \gamma_d}{d}.$$
This in turn means that for any $d$ and $r \in (0,1)$ we can produce a $\mathbb{Z}^d$-action with critical dimension $r$.

\subsubsection*{Dependence on the choice of summing set}

Consider a $\Zto$-action, constructed via the method above, and it's critical dimension with respect to $[-n,n] \times [-\floor{e^n-1}, \floor{e^n-1}]$. Here $s_2$ grows strictly faster than $s_1$ and, with the sensible choice representatives, the critical dimension is seen to be $\gamma = \gamma_2$. This, taken with the last example, shows that the critical dimension very much depends on the choice of summing sequence. It also shows that critical dimensions of the factors can be deduced from those of the product action and vice-versa.

In fact, any desired weighting of the critical dimensions can be achieved. Let $t_i \in [0,1]$ such that ${t_1 + ... + t_d = 1}$, and take $s_i(n) = n$ if $t_i = 0$ and $s_i(n) = \floor{(e^{n} - 1)^{t_i}}$ otherwise. Then the critical dimension of the product action with respect to corresponding summing sequence is given by $\gamma = t_1 \gamma_1 + ... + t_d \gamma_d$. Moreover, each such summing sequence is rectangular, and so each of these weightings is an invariant of metric isomorphism.

\subsection{Extension to non-product measures}

In the last part we assumed that the measure on  $X = X_1 \times ... \times X_d$ was given by a product of measures on each $X_i$. In this part we will remove that assumption. We still let $\mathbb{Z}^d \curvearrowright (X,\mu)$ via the product action but $\mu$ is not necessarily a product measure. We consider the critical dimensions of the $T_i$ with respect to the projection measures $\mu_i = \mu \circ \pi_i^{-1}$, where $\pi_i(x) = x_i$. Then we show the following, which the author expected but was unable to find in the literature.

\begin{prop} \label{equiprod}
Let $\mathbb{Z}^d \curvearrowright (X,\mu)$ via the product action $(u, x) \mapsto (T_1^{u_1}x_1,...,T_d^{u_d}x_d)$, which is assumed to be non-singular and ergodic. Then $\mu \sim \mu_1 \otimes ... \otimes \mu_d$.
\end{prop}

In particular this means that each such action is metrically isomorphic to a product action with product measure, and for a rectangular summing sequence $B_n$ combining proposition \ref{prodCD} with corollary \ref{CDinvar} gives the following result, which implies theorem \ref{=case}.

\newpage
\begin{theorem} \label{CDprodact}
Let $\Zd$ act on a product measurable space $X$ with measure $\mu$ via a non-singular and ergodic product action, and $B_1 \subseteq B_2 \subseteq...$ be a rectangular summing sequence. Let $D \subseteq [1,d]$ such that for each $i \in D$ the function $s_i$ is a greatest element in $\lbrace s_1,...,s_d \rbrace$ with respect to $\lesssim$. Then
\begin{align*}
\frac{\sum_{i \in D} a_i \alpha_i}{\sum_{i \in D} b_i} \leq \alpha(B_n) \leq \beta(B_n) \leq \frac{\sum_{i \in D} b_i \beta_i}{\sum_{i \in D} a_i}
\end{align*}
where $\alpha_i$ and $\beta_i$ are the critical dimensions of $T_i$ with respect to $[-n,n]$ and the projection measures $\mu_i = \mu \circ \pi_i^{-1}$.
\end{theorem}

By induction it is enough to consider the case $d = 2$ to prove proposition \ref{equiprod}. For notational simplicity take $Y = X_1$ and $Z = X_2$ so that $X = Y \times Z$ (as measure spaces). Our strategy is to use the following result of Brown and Dooley, in our notation.

\begin{prop}[see \cite{BroDoo1}] \label{disprod}
With $\mu$, $\mu_1$ and $\mu_2$ as described above we have that $\mu \sim \mu_1 \otimes \mu_2$ if and only if there is a disintegration
$$\mu = \int_{Y} \mu^{y} d\mu_1(y)$$
such that for all $y,y' \in Y$ we have $\mu^y \sim \mu^{y'}$.
\end{prop}

We show that there is a set of $\mu_1$-measure $1$ for which the measures $\mu^y$ are equivalent. This is enough to apply the proposition and deduce the result. The proof of this claim will rely on the disintegration theorem below, which follows from the significantly more general theorem 453K in \cite{Frem4}.

\begin{theorem}
Let $A$ and $B$ be Radon spaces, $\lambda$ be a probability measure on $A$, $\pi : A \to B$ a measurable function and $\nu = \lambda \circ \pi^{-1}$. Then there exists a $\nu$-a.e unique family of probability measures $\lbrace \lambda^{b} \rbrace_{b \in B}$ on $A$ such that
\begin{enumerate}[\normalfont (i)]
\item For each Borel set $E \subset A$ the map $b \mapsto \lambda^b(E)$ is measurable.
\item For $\nu$-a.e. $b \in B$ we have $\lambda^b(A \setminus \pi^{-1}(b)) = 0$.
\item For every Borel function $f : A \to [0,\infty]$ we have 
$$\int_A f(a) \, d\lambda(a) = \int_B \int_{\pi^{-1}(a)} f(a) \, d\lambda^b(a) \, d\nu(b).$$
\end{enumerate}
\end{theorem}

\begin{proof}[Proof of \ref{equiprod}]
Since Polish spaces are Radon spaces we can take $A = X$, $B = Y$ and $\pi = \pi_1$ in the disintegration theorem and hence may write $$\mu = \int_{Y} \mu^{y} d\mu_1(y)$$ as in \ref{disprod}. Here $\mu^y(E) = \lambda^y(\lbrace y \rbrace \times E)$, and hence this collection must also be unique $\mu_1$ almost surely.

Let $D \subset Y$ and $E \subset Z$ be measurable. Then
\begin{align*}
	(\mu \circ T_1)(D \times E) 
= 	\int_{T_1D} \mu^y(E) \, d\mu_1(y) 
= 	\int_{D} \mu^{T_1y}(E) \, d(\mu_1\circ T_1)(y) 
\end{align*}
and
\begin{align*}
	(\mu \circ T_1)(D \times E) 
&= 	\int_{X} \ind_{D \times E} \frac{d \mu \circ T_1}{d \mu} \, d\mu \\
&= 	\int_Y \int_Z \ind_{D \times E}(y,z) \frac{d \mu \circ T_1}{d \mu}(y,z) \, d\mu^{y}(z) 	\, d\mu_1(y) \\
&= 	\int_Y \int_Z \ind_{D \times E}(y,z) \frac{d \mu \circ T_1}{d \mu}(y,z) \frac{d \mu_1 }{d \mu_1 \circ T_1}(y) \, d\mu^{y}(z) 	\, d(\mu_1 \circ T_1)(y) \\
\end{align*}
Fix $E$. Taking $$D = \left\lbrace y \in Y : \int_E \frac{d \mu \circ T_1}{d \mu}(y,z)  \, d\mu^{y}(z) > \frac{d \mu_1  \circ T_1}{d \mu_1}(y) \mu^{T_1y}(E) \right\rbrace$$
and combining this with the above shows that $\mu_1 \circ T_1(D) = 0$ and hence $\mu_1(D) = 0$. By reversing the inequality we may deduce that for each measurable $E$ 
$$\int_E \frac{d \mu \circ T_1}{d \mu}(y,z)  \, d\mu^{y}(z) = \frac{d \mu_1 \circ T_1}{d \mu_1 }(y) \mu^{T_1y}(E) \quad \mu_1\text{-a.s..}$$
Since $T_1$ is non-singular with respect to both $\mu$ and $\mu_1$ this shows that for $\mu_1$-a.e. $y \in Y$ we have $\mu^{T_1y}(E) = 0$ if and only if $\mu^{y}(E) = 0$.

Consider the collection $\mathcal{C}$ of finite unions of open balls with rational radii and centres in a countable dense subset of $Z$. This collection is countable, and hence we can find a set $Y'$ of full $\mu_1$-measure for which for all $y \in Y'$ and all $E \in 
\mathcal{C}$ we have $\mu^{T_1y}(E) = 0$ if and only if $\mu^{y}(E) = 0$. Now observe that with $y$ fixed the functions $\ind_{>0}(\mu^{T_1y}(E))$ and $\ind_{>0}(\mu^{y}(E))$ define measures on $Z$. For $y \in Y'$ these measures agree on $\mathcal{C}$ and by monotone convergence must agree on all open sets of $Z$. It follows that for all $y \in Y'$ they agree for all measurable $E$, i.e. for all $y \in Y'$ we have $\mu^{T_1y} \sim \mu^{y}$. Without loss of generality we may assume that $Y'$ is $T_1$-invariant.

Now consider the equivalence classes $M_y = \lbrace \tilde{y} \in Y' : \mu^y \sim \mu^{\tilde{y}} \rbrace$, where $y \in Y'$, which partition $M_y$. An argument using indicator functions, similar to the one above, can be used to show that each $M_y$ is measurable. Moreover, our conclusion above shows that each $M_y$ is $T_1$-invariant and by the ergodicity of $\mu_1$ (inherited from $\mu$) some (unique) $M_y$ has measure $1$. We can then apply the result the result of Brown and Dooley (with $Y$ replaced by $M_y$) to see that $\mu \sim \mu_1 \otimes \mu_2$.
\end{proof}

\section{Further Questions}

Underlying much of this paper is the question of how the choice of summing sequence affects not only the critical dimension but the ergodic theorem for $\Zd$. On the one hand, for the sequences $[0,n]^d$ in $\Zd$ with $d > 1$ there is the counterexample to the ratio ergodic theorem \cite{Kren1}, found by Brunel and Krengel. On the other, for balls of norms or for rectangular summing sequences the ergodic theorem holds. If the sets in a summing sequence have the Besicovitch property and the modified doubling condition then it seems likely that Hochman's method will work, so long as some analogue of the finite coarse dimension property can be found. It is in proving this latter condition that both cases make use of some natural structure of $\Zd$. It would be interesting to know exactly what we require from a summing sequence in $\Zd$ for the ergodic theorem to hold. The fact that large parts of Hochman's approach can be applied to rectangles suggests that the theorems for norms and rectangles may both be special cases of a wider phenomenon.

On the critical dimension, we have shown in the case of product actions that the critical dimension for rectangles can be decomposed into a weighted average of the critical dimensions, for the projected measures, of maps corresponding to $e_1,...,e_n$. It is an open question whether this extends more generally, for example the critical dimension of each $e_i$ can be calculated on $(X,\mu)$ as a $\mathbb{Z}$-action regardless of whether the $\Zd$-action is a product action. Therefore it is reasonable to ask how the critical dimension of the $\Zd$-action is related to those of the generators.

\medbreak
\flushleft\textit{Acknowledgements.}
K.\ Jarrett was supported by an EPSRC studentship for the duration of this work. 
\medbreak

\end{document}